\documentclass[a4paper,11pt]{article}
\usepackage{amsfonts}
\usepackage{amsmath}
\usepackage{amsthm}
\usepackage{amssymb}
\newtheorem{theorem}{Theorem}[section]
\newtheorem{lemma}[theorem]{Lemma}

\theoremstyle{definition}

\theoremstyle{corollary}
\newtheorem{corollary}[theorem]{Corollary}

\theoremstyle{remark}
\newtheorem{remark}[theorem]{\bf Remark}

\numberwithin{equation}{section}
\newtheorem{problem}{Problem}[section]
\newcommand{\br}{\mathbb R}
\newcommand{\bn}{\mathbb N}

\title{Non-classical heat conduction problem with non
local source}
\author{
{Mahdi Boukrouche\thanks{
Address : Lyon University, F-42023 Saint-Etienne,
Institut Camille Jordan CNRS UMR 5208,
23 rue  Paul Michelon 42023 Saint-Etienne Cedex 2, France.
Mahdi.Boukrouche@univ-st-etienne.fr}}
\and
Domingo A. Tarzia\thanks{
Address: Departamento de Matem\'atica-CONICET,
FCE, Univ. Austral, Paraguay 1950, S2000FZF Rosario, Argentina.
DTarzia@austral.edu.ar}
}

\date{}

\begin{document}
\maketitle \normalsize

\begin{abstract}
We consider the non-classical heat conduction equation, in the
domain $D=\br^{n-1}\times\br^{+}$,   for which the internal energy
supply depends on  an integral function in the time variable of 
the heat flux  on the boundary $S=\partial D$, with homogeneous Dirichlet boundary condition and an initial condition. The
problem is motivated by the modeling of temperature regulation in
the medium.
The solution to the problem is found using  a Volterra integral equation of second kind in the
time variable $t$ with a parameter in $\br^{n-1}$. The solution to this Volterra equation is the
heat flux
$(y, s)\mapsto V(y , t)= u_{x}(0 , y , t)$
 on $S$, which is an additional unknown  of  the considered problem.
We show that a unique local solution exists, which  can be
extended globally in time. Finally a one-dimensional case is studied with some simplifications, we obtain the
solution explicitly by using the Adomian method and we derive its
properties.

\bigskip

\noindent{\it Keywords} :  Nonclassical n-dimensional heat equation,
non local sources, Volterra integral equation, existence and
uniqueness of solution, integral representation of solution,
explicit solution in 1-dimensional case and its properties.

\smallskip

\noindent{\it 2010 Mathematics Subject Classification} :
35C15, 35K05, 35K20,  35K60, 45D05, 45E10, 80A20.
\end{abstract}

\maketitle

\renewcommand{\theequation}{1.\arabic {equation}}
\setcounter{equation}{0}
\section{Introduction}

Let's consider the domain $D$ and its boundary $S$
  defined by
      \begin{eqnarray}
      &&D=\br^{n-1}\times\br^{+}
     =\{ (x , y)\in \br^{n} : \quad x=
      x_{1} >0, \quad y=( x_{2}, \cdots, x_{n})\in \br^{n-1}\}, \qquad
      \\
      &&S=\partial D=\br^{n-1}\times\{0\}=  \{(x , y)\in \br^{n} : \quad
      x=0, \quad y \in \br^{n-1} \}.
    \end{eqnarray}

The aim of this paper is to study the  following the  problem \ref{pb} with
 the non-classical heat equation, in the domain
$D$ with non local source, for which the internal energy supply depends on
the integral $\int_{0}^{t} u_{x}(0 , y, s) ds$
   on the
boundary $S$.

 \begin{problem}\label{pb}
 Find the temperature $u$,  at $(x , y , t)$ such that it satisfies the
 following conditions
 \begin{eqnarray*}\label{eqchN}
 u_{t} - \Delta u &=& -F\left(\int_{0}^{t}
 u_{x}(0, y , s)ds\right), \quad x=x_{1} >0, \quad y\in \br^{n-1},
 \quad  t>0, \label{cNpb}\nonumber\\
            u(0, y ,  t)&=& 0, \quad y\in\br^{n-1},   \quad t>0, \nonumber\\
  u(x, y, 0)&=& h(x , y),   \qquad x>0,  \quad y\in\br^{n-1},\label{cNIpb}
    \end{eqnarray*}
\end{problem}
where $\Delta$ denotes the Laplacian in $\br^{n}$.
This problem is  motivated by the modeling of  temperature regulation in an isotropic
medium, with  non-uniform and non local sources that provide cooling or heating
system. According to the properties of the function $F$ with respect to the heat flow
$V(y, s)=u_{x}(0, y , s)$
 at the  boundary
$S$.  For example, assuming  that
  \begin{eqnarray}
V \,  {\cal F}(V)> 0, \quad \forall
V\neq 0, \quad {\cal F}(0)=0,
  \end{eqnarray}
 with
 \begin{eqnarray}\label{w1}
 {\cal F}(V(y , t))=  F\left(
  \int_{0}^{t}V(y , s) ds\right)
  \end{eqnarray}
then,  see
    \cite{cannon1984, carslaw59},
 the cooling source occurs  when   $V(y, t)>0$ and heating source occurs when
    $V(y, t)<0$.

 \smallskip
Some references on the subject are
\cite{MT-qam} where ${\cal F}(V)= F(V)$,
\cite{berrone, ce, tarzia-villa, villa} where  the following
semi-one-dimension of this nonlinear problem, have been considered.
The  non-classical one-dimensional heat equation in a slab with
fixed or moving boundaries was studied in
\cite{bor-dt2006,
bor-dt2010, bor-dt2010-2, salva}. More references on the subject can be found in
\cite{cannon1989,  GLASHOFF81, GLASHOFF82, KENMOCHI90,    KENMOCHI88}. To our knowledge, it is the
first time that the solution to a non-classical heat conduction of
the type of Problem \ref{pb} is given. Other
non-classical problems
can be found in \cite{BBKW2010}.

The goal of this paper  is to obtain in Section 2 the existence and the uniqueness of the global solution of
the non-classical heat conduction {\it Problem }\ref{pb}, which is given through a Volterra integral equation.
In Section 3 we obtain the explicit solution of the one-dimensional case of {\it Problem }\ref{pb}, with some simplifications,
which is obtained by using the Adomian method through a double induction principle.

We recall here the Green's function for the n-dimensional heat
equation with homogeneous Dirichlet's boundary conditions, given the
following expression \cite{FRIEDMAN, LADY68}

\begin{eqnarray}\label{G}
        G_{1}(x , y , t ; \xi , \eta , \tau)=
{\exp\left[-{\|y-\eta\|^{2}
                     \over 4(t-\tau)}\right]
              \over  \left(2\sqrt{\pi(t-\tau)}\right)^{n-1}}G(x , t , \xi,
\tau),
            \end{eqnarray}
   where 
   $G$ is the Green's function for
   the one-dimensional case given by
$$G(x , t , \xi, \tau)=  {e^{-{(x-\xi)^{2} \over 4(t-\tau)}}-
e^{-{(x+\xi)^{2}\over 4(t-\tau)}}\over
    2\sqrt{\pi(t-\tau)}} \qquad t>\tau.
    $$


%

\renewcommand{\theequation}{2.\arabic {equation}}
\setcounter{equation}{0}
  \section{Existence results }

In this Section, we give first in Theorem \ref{th3.1}, the integral
representation (\ref{intSol}) of the solution of the considered  Problem
\ref{pb}, but it depends on the heat flow $V $on the
boundary $S$, which
satisfies the  Volterra integral equation (\ref{Volera}) with
initial condition (\ref{ci}). Then we prove, in Theorem \ref{th3.2},
under some assumptions on the data, that there exists a unique
solution of the Problem \ref{pb}, locally in times which can be
extended globally in times.

  \begin{theorem}\label{th3.1}
  The integral representation of a solution of the
considedred Problem
  {\rm\ref{pb}} is given by the following expression
\begin{eqnarray}\label{intSol}
     u(x , y ,  t)= u_{0}(x , y , t)
     - \int_{0}^{t}{{\rm erf}\left({x\over
     2\sqrt{t-\tau}}\right)  \over (2\sqrt{\pi(t-\tau)})^{n-1}}
      \left[  \int_{\br^{n-1}} \exp\left[-{\|y-\eta\|^{2}\over
      4(t-\tau)}\right]{\cal F}(V(\eta , \tau)) d\eta\right]d\tau \quad
      \end{eqnarray}

where
$$
 {\rm erf}\left(\zeta\right)=\left({2\over \sqrt{\pi}}
\int_{0}^{\zeta}e^{-X^{2}} dX \right)
$$

is the error function, with
\begin{eqnarray}\label{ci}
u_{0}(x , y , t)=\int_{D} G_{1}(x ,
y , t ; \xi , \eta , 0) h(\xi ,
     \eta) d\xi d\eta
   \end{eqnarray}
and the heat flux $V(y , t)=u_{x}(0 , y ,
t)$ on the
surface $x=0$,
  satisfies the following Volterra integral equation
         \begin{eqnarray}\label{Volera}
         V(y , t)=  V_{0}(y , t)
         -  2\int_{0}^{t}  {1\over  (2\sqrt{\pi(t-\tau)})^{n}}
          \left[  \int_{\br^{n-1}} \exp\left[-{\|y-\eta\|^{2}\over
          4(t-\tau)}\right]{\cal F}(V(\eta , \tau)) d\eta\right]d\tau \qquad
          \end{eqnarray}
     in the variable $t>0$, with $y \in \br^{n-1}$ is a parameter
     and
\begin{eqnarray}\label{ci}
V_{0}(y , t)= \int_{D} G_{1,x}(0 , y , t ; \xi , \eta , 0) h(\xi ,
         \eta) d\xi d\eta,
     \end{eqnarray}
 where     the function $(y , t)\mapsto {\cal F}(V(y , t))$
 is defined by {\rm(\ref{w1})} 
 for
  $y\in \br^{n-1}$ and $t>0$.
     \end{theorem}

  \begin{proof}
As  the boundary condition in Problem {\rm(\ref{pb})} is
homogeneous,  we have from \cite{FRIEDMAN}

\begin{eqnarray}\label{U}
 u(x , y , t)&=&
  \int_{D} G_{1}(x , y , t ; \xi ,\eta , 0) h(\xi ,\eta) d\xi d\eta
     \nonumber\\
     &&+
     \int_{0}^{t}
     \int_{D} G_{1}(x , y , t ; \xi , \eta ,
     \tau)[-{\cal F}(V(\eta , \tau))] d\xi d\eta d\tau,
           \end{eqnarray}
and therefore
\begin{eqnarray}\label{a}
 u_{x}(x , y , t)&=&
  \int_{D} G_{1,x}(x , y , t ; \xi ,\eta , 0) h(\xi ,\eta) d\xi d\eta
     \nonumber\\
     &&+
     \int_{0}^{t}
     \int_{D} G_{1,x}(x , y , t ; \xi , \eta ,
     \tau)[-{\cal F}(V(\eta , \tau))] d\xi d\eta d\tau.
           \end{eqnarray}

From    {\rm(\ref{G})} (the definition of $G_{1}$)   by
derivation     with respect to $x$, then taking $x=0$ we obtain

\begin{eqnarray}\label{a3}
\int_{D}G_{1,x}(0 , y , t ; \xi , \eta ,\tau){\cal F}(V(\eta , \tau))d\xi
d\eta &=&\int_{\br^{n-1}}{{\cal F}(V(\eta , \tau)) e^{-{\|y-\eta\|^{2}\over
4(t-\tau)}} \over(t-\tau)^{{n+2\over 2}}(2\sqrt{\pi})^{n}}
 \left(\int_{0}^{+\infty}\xi e^{-{\xi^{2}\over 4(t-\tau)}}
d\xi\right)d\eta \nonumber\\
 &=&{2\over (2\sqrt{\pi(t-\tau)})^{n}}
 \int_{\br^{n-1}}{\cal F}(V(\eta , \tau)) e^{-{\|y-\eta\|^{2}\over
4(t-\tau)}} d\eta,
 \end{eqnarray}
 as
 $$\int_{0}^{+\infty}\xi e^{-{\xi^{2}\over 4(t-\tau)}} d\xi = 2(t-\tau).$$

Thus taking $x=0$ in {(\ref{a})} with {(\ref{a3})}  we get
 {(\ref{Volera})}.

\bigskip

Also by (\ref{G}) we obtain
\begin{eqnarray*}
\int_{D}G_{1}(x , y , t ; \xi , \eta
,\tau){\cal F}(V(\eta , \tau))d\xi d\eta ={1\over
(2(\sqrt{\pi(t-\tau)})^{n}}\times\\ \times
\int_{D} e^{{-\|y-\eta\|^{2}\over
4(t-\tau)}}\left[e^{-{(x-\xi)^{2}\over 4(t-\tau)}}- e^{-{(x+\xi)^{2}\over
4(t-\tau)}}\right]{\cal F}(V(\eta ,
\tau))d\xi d\eta
\nonumber\\
={1\over (2(\sqrt{\pi(t-\tau)})^{n}}\int_{\br^{+}}
\left[e^{-{(x-\xi)^{2}\over 4(t-\tau)}}- e^{-{(x+\xi)^{2}\over 4(t-\tau)}}\right]d\xi \int_{\br^{n-1}}
e^{{-\|y-\eta\|^{2}\over 4(t-\tau)}}{\cal F}(V(\eta , \tau))d\eta
\end{eqnarray*}
 and by using
\begin{eqnarray*}\label{eqA}
\int_{0}^{+\infty} e^{-(x-\xi)^{2}\over 4(t-\tau)} d\xi &=&
2\sqrt{t-\tau} \left(\int_{-\infty}^{0} e^{-X^{2}} dX
+\int_{0}^{{x\over 2\sqrt{t-\tau}}} e^{-X^{2}} dX  \right)
\nonumber\\
&=& \sqrt{\pi(t-\tau)}\left(1+  {\rm erf}\left({x\over
2\sqrt{t-\tau}}\right)\right)
\end{eqnarray*}


and
\begin{eqnarray*}\label{eqB}
\int_{0}^{+\infty} e^{-(x+\xi)^{2}\over 4(t-\tau)} d\xi &=&
2\sqrt{t-\tau} \left(\int_{0}^{+\infty} e^{-X^{2}} dX
 -\int_{0}^{{x\over 2\sqrt{t-\tau}}} e^{-X^{2}} dX  \right)
\nonumber\\
&=& \sqrt{\pi(t-\tau)}\left(1-  {\rm erf}\left({x\over
2\sqrt{t-\tau}}\right)\right)
\end{eqnarray*}
we get

\begin{eqnarray*}\label{GK1}
\int_{D}G_{1}(x , y , t ; \xi , \eta ,\tau){\cal F}(V(\eta , \tau))d\xi
d\eta = { {\rm erf}\left({x\over 2\sqrt{t-\tau}}\right) \over
         (2\sqrt{\pi(t-\tau)})^{n-1}} \int_{\br^{n-1}}
e^{-{\|y-\eta\|^{2}\over 4(t-\tau)}}{\cal F}(V(\eta , \tau))  d\eta.
\end{eqnarray*}
Taking this formula in (\ref{U}) we obtain (\ref{intSol}).
\end{proof}

To solve the Volterra integral equation (\ref{Volera}), we rewrite it in the suitable form

\begin{lemma}\label{lem3.2}
The  Volterra integral equation
{\rm(\ref{Volera})} can be rewrite in the following form
\begin{eqnarray}\label{SV}
V(y , t)&=& {1\over t(2\sqrt{\pi\, t})^{n} } \int_{\br^{+}}\xi
e^{-{\xi^{2}\over 4t}}\left(\int_{\br^{n-1}} e^{-{\|y-
\eta\|^{2}\over 4t}}h(\xi ,
\eta)d\eta\right)d\xi\nonumber\\
&&-{2\over (2\sqrt{\pi})^{n} } \int_{0}^{t}
    {1   \over (t-\tau)^{n/2}}
\int_{\br^{n-1}}
{\cal F}(V(\eta , \tau))  e^{-{\|y- \eta\|^{2}\over
4(t-\tau)}}d\eta d\tau.
 \end{eqnarray}
\end{lemma}

\begin{proof}
Using the derivative, with respect to $x$, of {\rm(\ref{G})}, then
taking $x=0$  and $\tau=0$, then    taking the new expression of
 $V_{0} (y,t)$  in the Volterra integral
equation {\rm(\ref{Volera})} we obtain  {\rm(\ref{SV})}.
  \end{proof}

\bigskip

\begin{theorem}\label{th3.2}
Assume that $h\in \mathcal{C}(D)$,  $F\in \mathcal{C}(\br)$ and locally
Lipschitz in $\br$, then there exists a unique solution of the
problem \ref{pb} locally in times which can be extended globally in
times.
  \end{theorem}

\begin{proof}
We know from Theorem  {\rm(\ref{th3.1})}  that, to prove the
existence and uniqueness of the solution {\rm(\ref{intSol})} of
Problem {\rm(\ref{pb})}, it is enough to solve the
Volterra integral
equation {\rm(\ref{SV})}. So we rewrite it again as follows

\begin{eqnarray}\label{SV2}
V(y , t)= f(y , t) +\int_{0}^{t}g(y , \tau , V(y , \tau))d\tau
 \end{eqnarray}
with

\begin{eqnarray}\label{Phi}
f(y , t)= {1\over t(2\sqrt{\pi\, t})^{n} } \int_{\br^{+}}\xi
e^{-{\xi^{2}\over 4t}}\left(\int_{\br^{n-1}} e^{-{\|y-
\eta\|^{2}\over 4t}}h(\xi , \eta)d\eta\right)d\xi
\end{eqnarray}
and

\begin{eqnarray}\label{Psi}
 g(t , \tau , y , V(y , \tau))=-{2 (t-\tau)^{-n/2}
 \over (2\sqrt{\pi})^{n} }  \int_{\br^{n-1}}
{\cal F}(V(\eta , \tau)  e^{-{\|y- \eta\|^{2}\over
4(t-\tau)}}d\eta.
 \end{eqnarray}

We have to check the conditions $H1$ to $H4$ in Theorem 1.1 page 87,
and $H5$ and $H6$ in Theorem 1.2 page 91   in \cite{MILLER}.

\noindent$\bullet$ The function $f$ is defined and continuous for
all $(y , t)\in \br^{n-1}\times\br^{+}$, so $H1$ holds.

\smallskip

\noindent$\bullet$ The function $g$ is measurable in $(t, \tau, y
, x)$ for $0\leq \tau \leq t < +\infty$, $x\in \br$, $y\in
\br^{n-1}$, and continuous in $x$ for all $(y , t ,
\tau)\in\br^{n-1}\times\br^{+}\times\br^{+}$, $g(y , t , \tau ,
x)=0$ if $\tau >t$, so here we need the continuity of

$$  V(\eta , \tau) \mapsto {\cal F}(V(\eta , \tau)) = F\left( \int_{0}^{\tau}  V(\eta , s) ds\right),$$

which follows from the hypothesis that $F\in {\cal C}(\br)$. So $H2$ holds.

\smallskip

\noindent$\bullet$ For all $k >0$ and all bounded set $B$ in $\br$,
 we have
\begin{eqnarray*}
|g(y , t , \tau , X)| &\leq & {2\over (2\sqrt{\pi})^{n}} \sup_{X\in
B}|{\cal F}(X)| (t-\tau)^{-{n/2}} \int_{\br^{n-1}} e^{-\|y-\eta\|^{2\over
4(t-\tau)}} d\eta \nonumber\\
&\leq & {2\over (2\sqrt{\pi})^{n}} \sup_{X\in B}|{\cal F}(X)|
(t-\tau)^{-{n/2}}(2\sqrt{\pi (t-\tau)})^{n-1}
\nonumber\\
&= & {1\over \sqrt{\pi}}  \sup_{X \in B} |{\cal F}(X)|{1\over \sqrt{
(t-\tau)}}
\end{eqnarray*}
 thus there exists a measurable function $m$ given by
\begin{eqnarray}\label{m}
 m(t , \tau)={1\over \sqrt{\pi}}  \sup_{X
\in B} |{\cal F}(X)|{1\over \sqrt{ (t-\tau)}}
\end{eqnarray}
such that
\begin{eqnarray}\label{H3}
|g(y , t , \tau , X)| \leq m(t , \tau) \quad \forall 0\leq \tau \leq
t\leq k,\quad X\in B
\end{eqnarray}
and satisfies
\begin{eqnarray*}
 \sup_{t\in [0 , K]}
 \int_{0}^{t} m(t , \tau) d\tau
&=& {1\over \sqrt{\pi}} \sup_{X\in B}|{\cal F}(X)| \sup_{t\in [0 , k]}
 \int_{0}^{t}
  {1\over \sqrt{t-\tau}} d\tau
\nonumber\\
&= & {1\over \pi}\sup_{X\in B}|{\cal F}(X)| \sup_{t\in [0 ,
k]}\left(-2\sqrt{
  (t-\tau)}|_{0}^{t}\right)
\nonumber\\
  &=&{1\over \pi}\sup_{X\in B}|{\cal F}(X)| \sup_{t\in [0 , k]}\sqrt{t}
  \leq {2\sqrt{k}\over \pi} \sup_{X\in
  B}|{\cal F}(X)| <\infty,
 \end{eqnarray*}
so  $H3$ holds.

\smallskip

\noindent$\bullet$ Moreover  we have also
\begin{eqnarray}\label{th1.2-p91}
\lim_{t \to 0^{+}}\int_{0}^{t}m(t , \tau)d\tau
&=&{1\over\sqrt{\pi}}\sup_{X \in B} |{\cal F}(X)| \lim_{t \to 0^{+}}
   \int_{0}^{t}{d\tau\over
\sqrt{t-\tau}}
\nonumber\\
&=&  {1\over\sqrt{\pi}}\sup_{X \in B} |{\cal F}(X)|\lim_{t \to
0^{+}}(2 \sqrt{t}) =0,
\end{eqnarray}
and
\begin{eqnarray}\label{th2.3-p97}
\lim_{t \to 0^{+}}\int_{T}^{T+t}m(t , \tau)d\tau =
{1\over\sqrt{\pi}}\sup_{X \in B} |{\cal F}(X)|\lim_{t \to 0^{+}}(2
\sqrt{t}) =0.
\end{eqnarray}

\smallskip

\noindent$\bullet$ For each compact subinterval $J$ of $\br^{+}$,
each  bounded set $B$
 in $\br^{n-1}$, and each $t_{0}\in \br^{+}$, we set
\begin{eqnarray*}
\mathcal{A}(t, y , V(\eta)) =|g(t , \tau ; y , V(\eta , \tau))-g(t_{0}
, \tau ; y , V(\eta , \tau))|.
 \end{eqnarray*}

\begin{eqnarray*}
\mathcal{A}(t, y , V(\eta))={2\over (2\sqrt{\pi})^{n}}\int_{J}\left|
 \int_{\br^{n-1}} e^{-{\|y- \eta\|^{2}\over 4(t-\tau)}}{{\cal F}(V(\eta , \tau))
 \over (t-\tau)^{-{n/2}}}
 -
 e^{-{\|y-\eta\|^{2}\over 4(t_{0}-\tau)}}{{\cal F}(V(\eta , \tau))\over
(t_{0} -\tau)^{-{n/2}}} d\eta \right| d\tau
 \end{eqnarray*}
as the function $\tau \mapsto V(\eta , \tau)$ is continuous then
$$  \tau \mapsto   \int_{0}^{\tau} V(\eta , s) ds$$
is ${\cal C}^{1}(\br)$
and  is
in the compact $B\subset \br$ for all $\eta\in \br^{n-1}$, so by the
continuity of ${\cal F}$ we get ${\cal F}(V(\eta , \tau))\subset {\cal F}(B)$, that is
there exists $M>0$ such that ${\cal F}(V(\eta , \tau))|\leq M$ for all
$(\eta , \tau)\in\br^{n-1}\times \br^{+}$. So

\begin{eqnarray*}
\sup_{V(\eta)\in \mathcal{C}(J , B)} \mathcal{A}(t, y , V(\eta)) \leq
{2M\over (2\sqrt{\pi})^{n}} \sup_{ V(\eta)\in \mathcal{C}(J , B)}
 \left|
  \int_{\br^{n-1}}
{e^{-{\|y-\eta\|^{2}\over 4(t-\tau)}}
\over\sqrt{(t-\tau)}^{n}}d\eta
  -
 \int_{\br^{n-1}}
{e^{-{\|y-\eta\|^{2}\over
4(t_{0}-\tau)}}\over\sqrt{(t_{0}-\tau)}^{n}} d\eta \right|
 \end{eqnarray*}
using that
\begin{eqnarray*}
\int_{\br^{n-1}}  \exp\left[-{\|y-\eta\|^{2}\over
       4(t-\tau)}\right]d\eta =  \left(2\sqrt{\pi(t-\tau)}\right)^{n-1}
            \end{eqnarray*}
we obtain
\begin{eqnarray*}
\sup_{V(\eta)\in \mathcal{C}(J , B)} \mathcal{A}(t, y , V(\eta)) \leq
{2M\over (2\sqrt{\pi})^{n}} \sup_{V(\eta)\in \mathcal{C}(J , B)}
 \left|
{ (2\sqrt{\pi(t-\tau)})^{n-1}  \over (\sqrt{t-\tau})^{n}}
  -
{ (2\sqrt{\pi(t_{0}-\tau)})^{n-1} \over
(\sqrt{t_{0}-\tau})^{n}}\right|
 \end{eqnarray*}
 thus
\begin{eqnarray*}
\sup_{V(\eta)\in \mathcal{C}(J , B)} \mathcal{A}(t, y , V(\eta)) \leq
{M\over \sqrt{\pi}} \sup_{V(\eta)\in \mathcal{C}(J , B)}
 \left|
{{\sqrt{t_{0}-\tau} -  \sqrt{t-\tau}}
 \over
\sqrt{(t-\tau)(t_{0}-\tau)} }\right|.
 \end{eqnarray*}
Thus we deduce that
\begin{eqnarray*}
 \lim_{t\to t_{0}}  \int_{J}
 \sup_{V(\eta)\in \mathcal{C}(J , B)}\mathcal{A}(t, y , V(\eta))d\tau =0.
 \end{eqnarray*}
So  $H4$ holds.

\smallskip

\noindent$\bullet$ For all compact $I\subset \br^{+}$, for all
function $\psi \in \mathcal{C}(I ,\br^{n})$, and all $t_{0}>0$,
\begin{eqnarray*}
|g(t , \tau ; \psi(\tau))-g(t_{0} , \tau ,\psi(\tau))|= {2\over
(2\sqrt{\pi})^{n}}\left|
 \int_{\br^{n-1}}{\cal  F}(\psi(\tau))\left(
 {e^{-{\|y- \eta\|^{2}\over 4(t-\tau)}}
 \over (t-\tau)^{{n/2}}}
 -
 {e^{-{\|y-\eta\|^{2}\over 4(t_{0}-\tau)}}\over
(t_{0} -\tau)^{{n/2}}} \right)d\tau \right|
 \end{eqnarray*}
as ${\cal F}\in \mathcal{C}(\br)$ and $\psi \in \mathcal{C}(I ,\br^{n})$ then
there exists a constant $M>0$ such that $|{\cal F}(\psi(\tau))|\leq M$ for
all $\tau \in I$. Then we obtain as for H4, that
\begin{eqnarray*}
\lim_{t\to t_{0}}\int_{I}|g(t , \tau ; \psi(\tau))-g(t_{0} , \tau
,\psi(\tau))|d\tau =0.
 \end{eqnarray*}
 So H5 holds.

\noindent$\bullet$ Now for each constant $K>0$ and each bounded set
$B\subset \br^{n-1}$ there exists a measurable function $\varphi$
such that
\begin{eqnarray*}
|g(y , t , \tau , x) - g(y , t , \tau , X)|\leq \varphi(t ,
\tau)|x-X|
 \end{eqnarray*}
 whenever $0\leq \tau \leq t \leq K$ and both $x$ and $X$ are in
 $B$. Indeed as $F$ is assumed locally Lipschitz function in $\br$
 there exists constant $L>0$ such that
\begin{eqnarray*}
|{\cal F}(x) - {\cal F}(X)|\leq L(\tau)|x - X|
  \quad \forall (x , X)\in B^{2}
 \end{eqnarray*}
 with $L(\tau)= L\tau$.   Then we have
\begin{eqnarray*}
|g(y , t , \tau , x) - g(y , t , \tau , X)|&=&
 {2\over (2\sqrt{\pi})^{n}}
 \left|
 \int_{\br^{n-1}}  (t-\tau)^{-{n/2}}  e^{-{\|y- \eta\|^{2}\over 4(t-\tau)}}
 ({\cal F}(x)- {\cal F}(X)) d\eta
 \right|
\nonumber\\
 &\leq&{2\over (2\sqrt{\pi})^{n}}\left(
 \int_{\br^{n-1}} e^{-{\|y- \eta\|^{2}\over 4(t-\tau)}}d\eta \right)
  (t-\tau)^{-{n/2}}
L \tau|x-X|
\nonumber\\
 &\leq&
 {L \tau \over \sqrt{\pi(t-\tau)}}|x-X|,
 \end{eqnarray*}

 then $\varphi(t ,\tau)=
 {L\tau\over \sqrt{\pi(t-\tau)}}$. We have also for each $t\in [0 , k]$
 the function  $\varphi\in L^{1}(0 , t)$ as a function of $\tau$ and
\begin{eqnarray*}
 \int_{t}^{t+l}\varphi(t +l, \tau) \tau d\tau
 &=& {L \over \sqrt{\pi}}  \int_{t}^{t+l}
 { \tau   d\tau\over \sqrt{t+l-\tau}}
 = {L \over \sqrt{\pi}}  \int_{l}^{0}(u^{2}     - t - l) du
 \nonumber\\
 &=&{L l\over \sqrt{\pi}} (l+t -{1\over 3})\to 0 \quad \mbox{with }
 l\to 0
\end{eqnarray*}
where $u= \sqrt{t+l-\tau}$.

  So H6 holds. All the conditions H1 to H6 are satisfied with
  (\rm{\ref{th1.2-p91})} and (\rm{\ref{th2.3-p97})}.

Thus from \cite{MILLER} (Theorem 1.1 page 87, Theorem 1.2 page 91
and  Theorem 2.3 page 97) there exists a unique local times solution
of the Volterra integral equation {\rm(\ref{Volera})} which can be
extended globally in times.
  Then the proof of this theorem is complete.
  \end{proof}

\bigskip

\renewcommand{\theequation}{3.\arabic {equation}}
\setcounter{equation}{0}
  \section{The  one-dimensional case of Problem \ref{pb}}
   \label{section D1}

       Let us consider  now the one dimensional case of Problem \ref{pb}  for  the temperature defined by

 \begin{problem}\label{pb1d}
 Find the temperature $u$ at $(x , t)$ such that it satisfies the
 following conditions
 \begin{eqnarray*}\label{eqchP1}
 u_{t} - u_{xx} &=& -F\left(\int_{0}^{t}
 u_{x}(0, s)ds\right), \qquad x>0, \quad
      t>0, \label{cpb1}\nonumber\\
            u(0,  t)&=& 0, \quad  t>0, \nonumber\\
  u(x, 0)&=& h(x),   \qquad x>0. \quad\label{cNIpb}
    \end{eqnarray*}
\end{problem}

Taking into account that
\begin{eqnarray}
 \int_{0}^{t} G(x , t, \xi, \tau) d\xi = erf\left({x\over 2\sqrt{t-\tau}}\right)
\end{eqnarray}
thus the solution of the problem \ref{pb1d} is given by

\begin{eqnarray}
 u(x , t) = u_{0}(x, t) - \int_{0}^{t}
 erf\left({x\over 2\sqrt{t-\tau}}\right)
 F\left(\int_{0}^{\tau} W(\sigma)d\sigma \right) d\tau
\end{eqnarray}
with
\begin{eqnarray}\label{u}
 u_{0}(x, t) = \int_{0}^{t} G(x, t, \xi, 0) h(\xi) d\xi
\end{eqnarray}
and $W(t)= u_{x}(0, t)$ is the the solution of the following
Volerra integral equation
\begin{eqnarray}\label{Volterra}
 W(t)= V_{0}(t)
-
\int_{0}^{t} {F\left(
  \int_{0}^{\tau} W(\sigma)d\sigma\right)
\over\sqrt{\pi(t-\tau)}} d\tau
\end{eqnarray}
where
\begin{eqnarray}\label{V0}
 V_{0}(t)= {1\over 2\sqrt{\pi}t^{3/2}}\int_{0}^{+\infty} \xi e^{-\xi^{2}/4t}h(\xi) d\xi
 = {2\over \sqrt{\pi t}}\int_{0}^{+\infty} \eta e^{-\eta^{2}}h(2\sqrt{t} \,\eta) d\eta.
\end{eqnarray}

For the particular case

\begin{eqnarray}\label{h{0}}
 h(x)=h_{0} >0  \mbox{ pour } x>0, \mbox{ and } F(W)= \lambda W  \mbox{ for } \lambda\in \br,
\end{eqnarray}

then we have
\begin{eqnarray}\label{U_{0}}
u_{0}(t , x)= h_{0} erf\left({x\over 2\sqrt{t}}\right)
\end{eqnarray}

and the integral equation (\ref{Volterra}) becomes
\begin{eqnarray}\label{Vol}
 W(t) = {h_{0}\over \sqrt{\pi t}}
 - \lambda \int_{0}^{t}{ \int_{0}^{\tau} W(\sigma) d\sigma\over \sqrt{\pi(t-\tau)}}d\tau
\end{eqnarray}

\begin{lemma}\label{L31}
 Assume {\rm(\ref{h{0}})} holds. The solution of problem {\rm\ref{pb1d}}  is given by
 \begin{eqnarray}\label{eq310}
  u(x, t) = h_{0}\,  erf\left({x\over 2\sqrt{t}}\right)
  - \lambda \int_{0}^{t} erf\left({x\over 2\sqrt{t-\tau}}\right) U(\tau)d\tau
 \end{eqnarray}
 where $U$ is given by
 \begin{eqnarray}\label{eq311}
 U(t)= {h_{0}\over \sqrt{\pi}} \int_{0}^{t} {g(\tau)\over \sqrt{t-\tau}}d\tau
 \end{eqnarray}

and $g$ is the solution of the Volterra integral equation
\begin{eqnarray}\label{Vol}
 g(t)= 1- {2\lambda\over \sqrt{\pi}}\int_{0}^{t} g(\tau) \sqrt{t-\tau} d\tau.
\end{eqnarray}
Moreover, the heat flux on $x=0$ is given by
\begin{eqnarray}\label{heatflux}
 u_{x}(0, t) = U'(t)=  {h_{0}\over \sqrt{\pi t}} - h_{0}\lambda   \int_{0}^{t} g(\tau)d\tau, \quad t>0.
\end{eqnarray}
 \end{lemma}
 \begin{proof}
  We set
  \begin{eqnarray}
   U(t) = \int_{0}^{t}W(\tau) d\tau
  \end{eqnarray}
thus the function $U$  satisfies the following new Volterra integral equation
 \begin{eqnarray}\label{eq313}
  U(t)&=& 2h_{0} \sqrt{{t\over \pi}} -
  {\lambda\over \sqrt{\pi}}
  \int_{0}^{t}\int_{0}^{\tau}{U(\sigma)\over\sqrt{\tau-\sigma}}d\sigma d\tau\nonumber\\
  &=& 2h_{0} \sqrt{{t\over \pi}} - {2\lambda\over \sqrt{\pi}}\int_{0}^{t} U(\tau) \sqrt{t-\tau} d\tau,   \quad t> 0
 \end{eqnarray}
by using the following equality
\begin{eqnarray}
 \int_{\sigma}^{t} {d\tau\over \sqrt{\tau-\sigma}} = 2\sqrt{t-\sigma},  \quad 0< \sigma < t.
\end{eqnarray}
From [\cite{Belo}, p.229], the solution $t\mapsto U(t)$ of the integral equation (\ref{eq313})
is given by (\ref{eq311}) where $g$ is the solution of the Volterra equation (\ref{Vol}).

From  (\ref{Vol})  we obtain that
\begin{eqnarray}\label{eq316}
 \int_{0}^{t}  {g(\tau)\over \sqrt{t-\tau}} d\tau = 2\sqrt{t}
 - \lambda \sqrt{\pi}\int_{0}^{t}g(\tau)\sqrt{t-\tau}d\tau
\end{eqnarray}
using the following equality
\begin{eqnarray}\label{eq317}
 \int_{\sigma}^{t} {\sqrt{\tau-\sigma}\over \sqrt{t-\tau}}d\tau
 &=&(t-\sigma)\int_{0}^{1}{\sqrt{\xi}\over\sqrt{1-\xi}} d\xi
 =(t-\sigma)B({3\over 2} , {1\over 2})
 \nonumber\\
 &=&(t-\sigma) {\Gamma({3\over 2}) \Gamma( {1\over 2})\over \Gamma(2)}
 ={\pi \over 2}(t-\sigma)
\end{eqnarray}
where $B$ and $\Gamma$ are the classical Beta and Gamma functions defined below.

Therefore, we have that
\begin{eqnarray}\label{319}
 U(t) = 2h_{0}\sqrt{{t\over\pi}} - \lambda h_{0} \int_{0}^{t} g(\tau) (t-\tau) d\tau
\end{eqnarray}
and then the heat flux on $x=0$ is given by $u_{x}(0, t) = W(t)= U'(t)$,
that is (\ref{heatflux}) holds.
 \end{proof}

We recall here the well known  Beta an Gamma functions defined respectively by
\begin{eqnarray*}
 B(x, y) &=& \int_{0}^{1} t^{x-1} (1-t)^{y-1} dt, \quad x>0, \quad y>0, \nonumber\\
 \Gamma(x) &=& \int_{0}^{+\infty} t^{x-1} e^{-t} dt, \quad x>0,
\end{eqnarray*}
We will use in the next theorem the well known relations
\begin{eqnarray*}
 B(x , y) ={\Gamma(x)\Gamma(y)\over \Gamma(x+y)}, \quad \Gamma(x+1)=x\Gamma(x) \quad\forall x>0,
 \quad \Gamma({1\over 2})= \sqrt{\pi}, \quad \Gamma(n+1)=n! \quad\forall n\in \bn,
\end{eqnarray*}
and in particular the following one

\begin{lemma} For all integer $n \geq 1$ we have
\begin{eqnarray*}
 \Gamma(n+ {1\over 2})=  {(2n-1)!!\over 2^{n}} \sqrt{\pi},
  \end{eqnarray*}
 and we use the definition
 $$
 (2n-1)!!= (2n-1)(2n-3)(2n-5) \cdots 5\cdot 3\cdot 1
 $$
 for compactness  expression.
\end{lemma}

\begin{proof}
 For $n=1$ we get $\Gamma({3\over 2})= {\sqrt{\pi}\over 2}$ which is true. By induction
 we obtain  that
 \begin{eqnarray*}
 \Gamma\left(n+1+{1\over 2}\right)
 &=& \Gamma\left((n+{1\over 2})+1\right)
 =  \left(n +{1\over 2}\right) \Gamma\left(n +{1\over 2}\right)
 \nonumber\\
 &=&  \left(n +{1\over 2}\right) {(2n-1)!!\over 2^{n} }\sqrt{\pi}
 ={(2n+1)!!\over 2^{n+1}} \sqrt{\pi},
 \end{eqnarray*}
 thus the lemma is true.
\end{proof}

\begin{corollary}\label{cor}
 For all integer $n \geq 0$  we have also
 \begin{eqnarray*}
  \Gamma\left(3n+5 +{1\over 2}\right)= {(6n+9)!!\over 2^{3n+5}}\sqrt{\pi}
 \end{eqnarray*}
  \begin{eqnarray*}
   B\left({3\over 2}, 3n+4\right)
   =
   {\Gamma\left({3\over 2}\right)\Gamma(3(n+1)+1)\over \Gamma\left(3n+5+{1\over 2}\right)}
   ={(3(n+1))! \, 2^{3(n+1)+1}\over (6n+9)!!}
  \end{eqnarray*}
 \begin{eqnarray*}
   B\left({3\over 2}, 3n+{5\over 2}\right)
   =
   {\Gamma\left({3\over 2}\right)\Gamma\left(3n+ {5\over 2}\right)
   \over \Gamma(3(n+1)+1))}
   ={\pi (6n+3)!! \over  (3(n+1))! \,  2^{3(n+1)}   }
 \end{eqnarray*}
which will be useful in the next Lemma.
 \end{corollary}

We need previously some preliminary simple results in order
to obtain the solution of the integral equation (\ref{intSol})

\begin{eqnarray}
\int_{0}^{t} \sqrt{t-\tau} d\tau = {2\over 3}t^{3/2},
\qquad \int_{0}^{t}\tau^{3/2} \sqrt{t-\tau} d\tau = {\pi\over 2^{4}}t^{3},
\end{eqnarray}
 \begin{eqnarray}
 \int_{0}^{t}\tau^{3} \sqrt{t-\tau} d\tau ={2^{4} \, 3!\over 9!!}t^{9/2},
 \qquad
  \int_{0}^{t}\tau^{9/2} \sqrt{t-\tau} d\tau ={\pi \, 9!! \over 2^{6} 6!}t^{6},
 \end{eqnarray}

 \begin{eqnarray}
 \int_{0}^{t}\tau^{6} \sqrt{t-\tau} d\tau ={2^{7} \, 6!\over 15!!}t^{15/2},
 \qquad
  \int_{0}^{t}\tau^{15/2} \sqrt{t-\tau} d\tau ={\pi \, 15!! \over 2^{9} 9!}t^{9},
 \end{eqnarray}

 which can be generalized by the following ones:

 \begin{lemma}\label{Lem4}
For all integer $n \geq 0$   we have
 \begin{eqnarray}\label{323}
   \int_{0}^{t}\tau^{2n+3} \sqrt{t-\tau} d\tau ={2^{3n+4} \, (3(n+1))!\over (6n+9)!!}t^{3(2n+3)/2},
 \end{eqnarray}
\begin{eqnarray}\label{324}
   \int_{0}^{t}\tau^{3(2n+1)\over 2} \sqrt{t-\tau} d\tau
   ={\pi \, (6n+3)!! \over 2^{3(n+1)} (3(n+1))!}t^{3(n+1)}.
 \end{eqnarray}
 \end{lemma}
\begin{proof}
 Taking the change of variable $\tau= t\xi$ in (\ref{323}) using Corollary \ref{cor} we get
 \begin{eqnarray*}
   \int_{0}^{t}\tau^{2n+3} \sqrt{t-\tau} d\tau
  &=& t^{3(2n+3)\over 2}\int_{0}^{1}\xi^{3n+3} (1-\xi)^{1\over 2} d\xi
   =t^{3(2n+3)\over 2}\int_{0}^{1}\xi^{(3n+4)-1}(1-\xi)^{{3\over 2} -1}d\xi
  \nonumber\\
  &=& t^{3(2n+3)\over 2} B({3\over 2} , 3n+4)
  = {(3(n+1))! \, 2^{3(n+1) +1}\over (6(n+9)!!},
 \end{eqnarray*}
and
\begin{eqnarray*}
   \int_{0}^{t}\tau^{3(2n+1)\over 2} \sqrt{t-\tau} d\tau
   &=&  t^{3(2n+1)\over 2} \int_{0}^{1} \xi^{(3n+{5\over 2})-1}(1-\xi)^{{3\over 2}-1}d\xi
    \nonumber\\
   &=&t^{3(2n+1)\over 2} B\left({3\over 2} , 3n+{5\over 2}\right)
   \nonumber\\
   &=& {\pi \, (6n+3)!! \over (3(n+1))! \, 2^{3(n+1)}}   t^{3(n+1)},
 \end{eqnarray*}
thus the (\ref{323})-(\ref{324}) hold.
\end{proof}

Now, we will obtain the explicit solution of the integral equation :
\begin{eqnarray}\label{iq}
 y(t)= 1- {2\lambda\over \sqrt{\pi}}\int_{0}^{t} y(\tau) \sqrt{t-\tau} d\tau, \quad t>0,
\end{eqnarray}
by using Adomian decomposition method \cite{Adm, Adl,  waz1} through a serie expansion.

\begin{theorem}\label{th3.5}
 The solution of the integral equation {\rm(\ref{iq})} is given by the following expression
 \begin{eqnarray}\label{337}
  y(t) = I(t) - \sqrt{{2\over\pi}} J(t),  \quad t>0,
 \end{eqnarray}
with
 \begin{eqnarray}\label{338}
  I(t) = \sum_{n=0}^{+\infty}{(\lambda^{2/3}t)^{3n}\over (3n)!}
 \end{eqnarray}
 and
 \begin{eqnarray}\label{339}
  J(t) =  \sum_{n=0}^{+\infty}{(\lambda^{2/3}t)^{{3(2n+1)\over 2}}\over (3(2n+1))!!}
 \end{eqnarray}
 are series with infinite radii of convergence.
\end{theorem}
\begin{proof}
 Following the idea of \cite{
 Adm1, Boug, Chen, Hoss,  Sidd,  waz,
 waz2, waz3, wu}
 we propose, for the solution of the integral equation
 {\rm(\ref{iq})}, the following serie of expansion  functions given by
  \begin{eqnarray}\label{340}
   y(t) = \sum_{n=0}^{+\infty} y_{n}(t),
  \end{eqnarray}
and we obtain the following recurrence expressions
\begin{eqnarray}\label{341}
 y_{0}(t)=1,    \qquad y_{n}(t)= -{2\lambda\over \sqrt{\pi}}
 \int_{0}^{t}   y_{n-1}(\tau) \sqrt{t-\tau} \, d\tau,  \quad \forall n \geq 1.
\end{eqnarray}

Then we get

\begin{eqnarray}\label{342}
 y_{1}(t)=-{2\lambda\over \sqrt{\pi}}
 \int_{0}^{t}    \sqrt{t-\tau} \, d\tau = -{4\lambda\over3 \sqrt{\pi}}t^{3/2}
 = -\sqrt{{2\over\pi}} {\left(2\lambda^{2/3}t\right)^{3/2}\over 3!!},
\end{eqnarray}

\begin{eqnarray}\label{343}
  y_{2}(t)=-{2\lambda\over \sqrt{\pi}}
 \int_{0}^{t}   (-{4\lambda\over 3 \sqrt{\pi}}\tau^{3/2}) \sqrt{t-\tau} \, d\tau
 = {8\lambda^{2}\over3 \pi}  \int_{0}^{t} \tau^{3/2}\sqrt{t-\tau} \, d\tau
 = {\lambda^{2}t^{3}\over 3!}.
\end{eqnarray}

The first step of the double induction principle is just verified by
(\ref{337}) taking into account (\ref{342}) (\ref{343}). The second
step, we suppose by induction hypothesis that we have

\begin{eqnarray}\label{344}
 J_{2n}(t) = {\lambda^{2n}\over (3n)!}t^{3n}, \qquad
 J_{2n+1}(t)
 = - {2^{3n+2}\over (3(2n+1))!!}{\lambda^{2n+1}\over \sqrt{\pi}} t^{{3(2n+1)\over 2}}.
\end{eqnarray}

Therefore,  we obtain
\begin{eqnarray}\label{345}
 J_{2n+2}(t) &=& -{2\lambda\over\sqrt{\pi}}\int_{0}^{t} y_{2n+1}(\tau)\sqrt{t-\tau} d\tau
 = {\lambda^{2n+2}\over \pi}{2^{3n+3}\over (6n+3)!!}
 \int_{0}^{t} \tau^{3(2n+1)\over 2}\sqrt{t-\tau} d\tau
 \nonumber\\
 &=& {\lambda^{2n+2}\over \pi} {2^{3(n+1)}\over (6n+3)!!} {\pi \over 2^{3(n+1)}}
 {(6n+3)!!\over (3(n+1))!} t^{3(n+1)}
  \nonumber\\
  &=&
  {\lambda^{2n+2}\over (3(n+1))!} t^{3(n+1)}
\end{eqnarray}

and

\begin{eqnarray}\label{346}
Y_{2n+3}(t) &=& -{2\lambda\over\sqrt{\pi}}
\int_{0}^{t} y_{2n+2}(\tau)\sqrt{t-\tau}d\tau
 =  -{2\lambda^{2n+3}\over(3(n+1))! \sqrt{\pi}}
   \int_{0}^{t} \tau^{3n+3}\sqrt{t-\tau} d\tau
   \nonumber\\
 &=& -{2\lambda^{2n+3}\over(3(n+1))! \sqrt{\pi}}
    {2^{3n+4}(3(n+1))!\over (6n+9)!!} t^{3(2n+3)\over 2}
     \nonumber\\
  &=& -{2^{3(n+1)+2} \lambda^{2(n+1)+1}\over \sqrt{\pi} (3(2n+1)+1))!!} t^{{3(n+1)+1}\over 2}
\end{eqnarray}
This ends the proof.
\end{proof}

 \begin{remark}\label{rem1}
Taking $t\to 0^{+}$ in  (\ref{eq313}), (\ref{heatflux}), and (\ref{Vol}),
we obtain
 \begin{eqnarray*}\label{320}
 &&  W( 0^{+})= +\infty, \quad  W'( 0^{+})= -\infty,
   \nonumber\\
 && U( 0^{+}) = 0,  \quad U'( 0^{+}) =+\infty,
  \nonumber\\
 &&  g( 0^{+}) = 1, \quad g'(0^{+}) = 0.
 \end{eqnarray*}
 So we deduce that the heat flux $W$ and the total heat flux $U$, and also $g$
 are positive functions in a neighbourhood of $t=0$.
 \end{remark}

  \bigskip
\noindent{\bf Conclusion:} We have obtained the global solution of a non-classical heat conduction problem in a
semi-n-dimensional space. Moreover, for the one-dimensional case  we have obtained the explicit solution by using
the Adomian method with a double induction principle.

\bigskip

\noindent{\bf Competing interests:}
The authors declare that they have no competing interests.

\bigskip

\noindent{\bf Author's contributions:}
The authors declare that the work was realized in collaboration with the same responsibility. All authors read and
approved the final manuscript.

\bigskip

\noindent{\bf Acknowledgements:}
 This paper was partially sponsored by the Institut Camille Jordan St-Etienne University for first author,
 and the projects PIP $\#$ 0534 from CONICET-Austral (Rosario, Argentina) and Grant AFOSR-SOARD FA 9550-14-1-0122
 for the second author.

\newpage

    \end{document}